\documentclass{amsart}
\usepackage{amsmath,amssymb}

\newtheorem{theorem}{Theorem}

\theoremstyle{definition}

\theoremstyle{remark}
\newtheorem{remark}[theorem]{Remark}
\numberwithin{equation}{section}

\begin{document}

\title[$L^p$-Green Potential Estimates on Noncompact Riemannian Manifolds]{Coarse and Precise $L^p$-Green
Potential Estimates on Noncompact Riemannian Manifolds\ $^\dag$}


\thanks{$^\dag$ Supported in part by Natural Science and
Engineering Research Council of Canada.}



\author{Jie Xiao}
\address{Department of Mathematics and Statistics, Memorial University of Newfoundland, St. John's, NL A1C 5S7, Canada}
\email{jxiao@mun.ca}


\date{}



\begin{abstract}
We are concerned about the coarse and precise aspects of a priori
estimates for Green's function of a regular domain for the
Laplacian-Betrami operator on any $3\le n$-dimensional complete
non-compact boundary-free Riemannian manifold through the square
Sobolev/Nash/logarithmic-Sobolev inequalities plus the rough and
sharp Euclidean isoperimetric inequalities. Consequently, we are led
to evaluate the critical limit of an induced monotone Green's
functional using the asymptotic behavior of the Lorentz norm deficit
of Green's function at the infinity, as well as the harmonic radius
of a regular domain in the Riemannian manifold with nonnegative
Ricci curvature.
\end{abstract}
\maketitle

\section{Introduction}

To highlight the key issues around the current paper, let us recall
a few of background materials as follows.

Let $M^n$ be a complete Riemannian manifold of dimension $n\ge 3$
and its length element take in local coordinates the following form
$$
ds^2=\sum_{j,k=1}^n g_{jk}(x)dx^jdx^k,
$$
where $(g_{jk})$ is a symmetric positive definite matrix leading:
the inverse matrix $(g_{jk})^{-1}=(g^{jk})$; the determinant
$g=\hbox{det}(g_{jk})$; the Riemannian volume element
$dV(x)=\sqrt{\hbox{det}(g_{jk})}dx$; and the distance between two
points $x,y\in M^n$: $d(x,y)=\inf_\gamma L(\gamma)$ in which the
infimum is taken over all piece-wise $C^1$-curves $\gamma$ in $M^n$
with $L(\gamma)$ being defined as the sum of the lengths
$$
L(\gamma_i)=\int_0^1|\gamma_i'(t)|\,dt;\quad \gamma_i(0)=x_i\
\hbox{and}\ \gamma_i(1)=y_i
$$
of the $C^1$ pieces $\gamma_i: [0,1]\mapsto M^n$ making $\gamma$.
One way of measuring the degree to which the geometry decided by
$ds^2$ might be different from $dx^2$ of the Euclidean space
$\mathbb R^n$ is the Ricci curvature tensor $Ric(u,v)$ -- this
tensor, acting on two vectors $u$ and $v$ in the tangent space
$T_pM^n$ of $M^n$ at the point $p\in M^n$, is defined as the trace
of the linear transformation $w\mapsto R(w,v)u$, a self-mapping of
$T_pM^n$, where $R(\cdot,\cdot)$ stands for the Riemann curvature
tensor. Since $Ric(u,v)=Ric(v,u)$, the Ricci tensor is completely
determined by $Ric(u,u)$ for all vectors $u$ of unit length -- this
function on the family of unit tangent vectors is usually called the
Ricci curvature -- when $Ric(u,u)$ is nonnegative for all unit
tangent vectors $u$ then $(M^n,ds^2)$ is said to be have nonnegative
Ricci curvature.

Next, denote by $\nabla$ and $\Delta$ the gradient and
Lapacian-Bertrami operators which are regarded as the most
fundamental operators on $(M^n,ds^2)$ and determined respectively in
local coordinates via:
$$
\nabla=\Big(\sum_{k=1}^n
g^{1k}\frac{\partial}{\partial x_k},...,\sum_{k=1}^n
g^{nk}\frac{\partial}{\partial x_k}\Big)
$$
and
$$
\Delta =-\frac{1}{\sqrt{g}}\sum_{j,k=1}^n\frac{\partial}{\partial
x_j}\Big(\sqrt{g}g^{jk}\frac{\partial}{\partial x_k}\Big).
$$
Associated with these two operators are the forthcoming three curvature-free quantities:
capacity, heat kernel and Green's function. If $C^\infty_0(E)$ means the class of all $C^\infty$ smooth functions
compactly supported in a set $E\subseteq M^n$, then the harmonic/Newtonian/Wiener capacity of a pre-compact open
set $\Omega\subseteq M^n$ is defined by
$$
cap_2(\Omega)=\inf\Big\{\int_{\Omega}|\nabla f|^2\,dV:\quad f\in
C^\infty_0(M^n),\ f\ge 1\ \hbox{on}\ \Omega\Big\}.
$$
Usually, $f$ is called a ($1\le p<\infty$, homogeneous) Sobolev
function on $\Omega$ provided $\int_{\Omega}|\nabla
f|^p\,dV<\infty$. Moreover, the notation $p_t(x,y)$ is used as the
heat kernel on $M^n$ -- that is -- the smallest positive solution
$u(t,x,y)$ to the heat equation
\[
\left\{\begin{array} {r@{\;,\quad}l} \frac{\partial}{\partial t}
u(t,x,y)=\Delta u(t,x,y)& (t,x,y)\in(0,\infty)\times M^n\times M^n,\\
\lim_{t\to 0}u(t,x,y)=\delta_x(y) & (x,y)\in M^n\times M^n,
\end{array}
\right.
\]
where $\delta_x(y)$ is the Dirac measure. At the same time, the
symbol $G(x,y)=\sup_{\Omega} G_\Omega(x,y)$ is chosen for the
minimal Green's function of $M^n$ where the supremum is taken over
all regular (meaning that the Dirichlet problem is uniquely
solvable) bounded open sets $\Omega\subseteq M^n$ containing $x$,
and $G_\Omega(x,y)$ is the nonnegative Green's function of $\Omega$
obeying
\[
\left\{\begin{array} {r@{\;,\quad}l}
\Delta G_\Omega(x,y)=\delta_x(y) & y\in\Omega,\\
G_\Omega(x,y)=0 & y\in M^n\setminus\Omega.
\end{array}
\right.
\]
Of course, the last equation on $\Delta$ is understood in the sense
of distribution, i.e.,
$$
\int_\Omega \langle \nabla G_\Omega(x,\cdot),\nabla \phi\rangle
dV=\phi(x)\quad\hbox{for\ all\ functions}\ \phi\in
C_0^\infty(\Omega),
$$
where $\langle\cdot,\cdot\rangle$ stands for the Riemannian metric
on $M^n$. The above-considered manifolds can be classified into two
types: non-parabolic or parabolic according to that $G(x,y)$ is
finite or infinite for all $y\in M^n\setminus\{x\}$ or all $y\in
M^n$. It is interesting, natural and important that the classical
relation between the heat kernel and the minimal Green's function
for a non-parabolic manifold is determined by
$$
G(x,y)=\int_0^\infty p_t(x,y)\,dt,\quad (x,y)\in M^n\times M^n.
$$
For more closely-related information see also Varopolulos \cite{Va},
Li-Yau \cite{LiYau}, Li-Tam \cite{LiT}, Li-Tam-Wang \cite{LiTW}, Li
\cite{LiP1}, Schoen \cite{Sch} and Davies-Safarov \cite{DaSa}.

More than that, however, we will employ $X\lesssim Y$ (i.e.,
$Y\gtrsim X$) to represent that there is a constant $\kappa>0$ such
that $X\le\kappa Y$, but also make $X\approx Y$ stand for both
$X\lesssim Y$ and $Y\lesssim X$.

With the help of the previously-described notions and notations, in
the subsequent Sections 2-3 we will, coarsely and precisely, treat a
series of monotonic $L^p ({0\le p<\frac{n}{n-2}})$-estimates for
Green's functions associated with Laplacian-Betrami operators on
non-compact complete boundary-free Riemannian manifolds with
dimension $n\ge 3$ through three mutually-implied inequalities --
$L^2$-Sobolev, -Nash and -logarithmic Sobolev inequalities.

In Theorem \ref{thm1} we will revisit five known characterizations
(Faber-Krahn's eigenvalue inequality, Maz'ya iso-capacitary
inequality, On/Off-diagonal upper bound of heat kernel and
distributional behavior of the Green function of the manifold) of
the $L^2$-Sobolev/Nash/Log-Sobolev inequalities, and in Theorem
\ref{thm1+} we give two more new ones -- Lorentz-Green local
inequality (i.e., distributional behavior of the Green function of
any regular bounded open subset of the manifold) and the following
Green potential-volume inequality:
$$
\sup_{x\in\Omega}\int_\Omega G_\Omega(x,y)\,dV(y)\lesssim
V(\Omega)^\frac{2}{n}.
$$
In the above and below, $\Omega$ is a regular bounded open set
$\Omega\subseteq M^n$ (containing $x$ whenever needed).
Interestingly, this last estimate leads to a discovery of the
comparison principle as stated in Theorem \ref{cor1}, especially
saying that if the coarse $L^1$-Sobolev inequality (equivalently,
the rough isoperimetric inequality of Euclidean type) is valid then
$$
\left(\int_\Omega
G_\Omega(x,y)^p\,dV(y)\right)^\frac{n}{n-p(n-2)}\lesssim
\left(\int_\Omega G_\Omega(x,y)^q\,dV(y)\right)^\frac{n}{n-q(n-2)}
$$
holds for $0\le q<p<\frac{n}{n-2}$, but also $G_\Omega(x,\cdot)$
belongs uniformly to the Lorentz space
$L^\ast_{\frac{n}{n-2}}(\Omega)$; see also Gr\"uter-Widman
\cite{GrWi} for an earlier result in $\mathbb R^n$.

In Theorems \ref{thm2} \& \ref{thm2+}, we will give the sharp
inequalities corresponding to those in Theorems \ref{thm1} \&
\ref{thm1+} under the assumption that $M^n$ allows the optimal
isoperimetric inequality of Euclidean type to exist. In particular,
we will establish such a classification result that together with
the nonnegative Ricci curvature and the volume $\omega_n$ of the
unit ball of $\mathbb R^n$, either the sharp Maz'ya's iso-capacitary
inequality
$$
V(\Omega)^{\frac{n-2}{n}}\le\big(n(n-2)\omega_n^\frac2n\big)^{-1}
cap_2(\Omega)
$$
or the sharp Green potential-volume inequality (cf. Weinberger
\cite{We} for the $\mathbb R^n$-case):
$$
\sup_{x\in\Omega}\int_\Omega G_\Omega(x,y)\,dV(y)\le
(2n\omega_n^\frac2n)^{-1}V(\Omega)^\frac2n
$$
ensures that $M^n$ is either isometric or diffeomorphic to $\mathbb
R^n$. Even more interestingly, this last inequality will drive us to
obtain the optimal monotone principle for the $L^p$-Green potentials
-- Theorem \ref{con2}: under $M^n$ being of the sharp isoperimetric
inequality of Euclidean type, the functional
$$
\left(\frac{\big(n(n-2)\omega_n^\frac2n\big)^p}{pB(\frac{n}{n-2}-p,p)}\int_\Omega
G_\Omega(x,y)^p\,dV(y)\right)^\frac{n}{n-p(n-2)}
$$
is monotone decreasing as $p$ varies from $0$ to $n/(n-2)$ (viewed
as the critical index) where $B(\cdot,\cdot)$ is the classical Beta
function. Consequently, we can evaluate the limit of the above
monotone functional as $p\to n/(n-2)$ -- this limit is fortunately
found to equal $\lim_{t\to\infty}D_\Omega(x,t)^\frac{n}{2-n}$ where
$$
D_\Omega(x,t)=V(\{y\in\Omega: G_\Omega(x,y)\ge
t\})^\frac{2-n}{n}-n(n-2)\omega_n^\frac2n t
$$
is the so-called Lorentz norm deficit of $G_\Omega(x,\cdot)$ at
$(x,t)$. Even more fortunately, if $(M^n,ds^2)$ has also nonnegative
Ricci curvature then
$$
\lim_{t\to\infty}D_\Omega(x,t)^\frac{n}{2-n}=\omega_n R_\Omega(x)^n,
$$
i.e., the volume of the geodesic ball with center $x$ and harmonic
radius $R_\Omega(x)$ of $\Omega$ at $x$ -- it seems worthwhile to
indicate that $R_\Omega^{2-n}(x)$ is the so-called Robin mass of
$\Omega$ at $x$ -- see also Bandle-Flucher \cite{BaFl}, Flucher
\cite{Flu} and Bandle-Brillard-Flucher \cite{BaBrFl} in regard to
this concept of Euclidean type and its applications in geometric
potential analysis -- and yet, it is perhaps appropriate to mention
one more fact that the quantity $R_\Omega(x)^{2-n}$ is the leading
term of the regular part of the Green function $G_\Omega(x,y)$ and
when $\Omega$ approaches $M^n$, the resulting harmonic radius at
$x$:
$$
\lim_{y\to x}\Big(d(x,y)^{2-n}-n(n-2)\omega_n
G(x,y)\Big)^\frac{1}{2-n}
$$
can be modified for the Green function of the Yamabe operator on a
compact Riemannian manifold in order to properly construct a
canonical metric on a locally conformally flat Riemannian manifold
-- see Habermann-Jost \cite{HaJo} (or Habermann \cite{Ha}) where the
construction of the metric deepens Hersch's \cite{Her} and
Leutwiler's \cite{Leu} and relies on Schoen-Yau's positive mass
theorem in \cite{SchYau1} with connection to Schoen's final solution
to Yamabe problem \cite{Sch0}. Needless to say, the endpoint
$n/(n-2)$ is crucial -- see also Schoen-Yau's paper \cite{SchYau}
for an account on the crucial role of such an exponent/index playing
in the integrability of the minimal Green's function for the
conformally invariant Laplace operator.

The proofs of the above-mentioned theorems (mixing Riemannian
geometry and integration theory of Green's functions) will be well
detailed except those already-known parts in Theorems \ref{thm1} and
\ref{thm2} (whose references will surely be located). Our ideas,
techniques and methods determining the sharp constants are
adjustable for settling the similar problems on the Green function
of the Yamabe operator -- this will be an object of our future
investigation.

We would like to take this opportunity to acknowledge several
communications on certain inequalities in this paper with: D. R.
Adams, C. Bandle, H. Brunner, W. Chen, P. Li, C. Xia, D. Yang, Z.
Zhai and G. Zhang, in alphabetical order.

\section{Coarse Estimates}

It is well-known that geometry of the rough $L^1$-Sobolev inequality
is completely characterized in terms of the isoperimetric inequality
with a rough constant -- that is to say -- the classic Sobolev
inequality
$$
\left(\int_{M^n}|f|^\frac{n}{n-1}\,dV\right)^\frac{n-1}{n}\lesssim \int_{M^n}|\nabla f|\,dV\quad\hbox{for\ all\ functions}\ f\in C_0^\infty(M^n),
$$
amounts to the generic isoperimetric inequality
$$
V(\Omega)^\frac{n-1}{n}\lesssim S(\partial\Omega)\quad\hbox{for\ all\ smooth\ bounded\ domains}\ \Omega\subseteq M^n.
$$
Here and henceforth, $S(\partial\Omega)$ is the area of the boundary
$\partial\Omega$ of $\Omega\subseteq M^n$.

While working on some analytic and geometric forms of $L^2$-Sobolev
inequality, -Nash inequality, and -logarithmic-Sobolev inequality
without optimal constant, we get a collection of the most-known
equivalences:

\begin{theorem}\label{thm1} Let $(M^n,ds^2)$, $n\ge 3$, be non-compact
complete boundary-free Riemannian manifold. Then the following eight
statements with coarse constants are equivalent:

\item{\rm(i)} $L^2$ Sobolev inequality
$$
\left(\int_{M^n}|f|^\frac{2n}{n-2}\,dV\right)^\frac{n-2}{n}\lesssim\int_{M^n}|\nabla
f|^2\,dV
$$
holds for all $f\in C^\infty_0(M^n)$.

\item{\rm(ii)} Nash's inequality
$$
\Big(\int_{M^n}|f|^2\,dV\Big)^{1+\frac{2}{n}}\lesssim\Big(\int_{M^n}|f|\,dV\Big)^\frac4n\int_{M^n}|\nabla
f|^2\,dV
$$
holds for all $f\in C^\infty_0(M^n)$.

\item{\rm(iii)} $L^2$ logarithmic Sobolev inequality
$$
\exp\left(\frac2n\int_{M^n}|f|^2\log |f|^2\,dV\right)\lesssim\int_{M^n}|\nabla f|^2\,dV
$$
holds for all $f\in C^\infty_0(M^n)$ with $\int_{M^n}|f|^2\,dV=1$.

\item{\rm(iv)} Faber-Krahn's eigenvalue inequality
$$
{\lambda_1(\Omega)}^{-1}=\sup_{f\in C^\infty_0(\Omega),\,
f\not\equiv 0}\frac{\int_\Omega |f|^2\,dV}{\int_\Omega|\nabla
f|^2\,dV}\lesssim V(\Omega)^{\frac2n}
$$
holds for all regular bounded open sets $\Omega\subseteq M^n$.

\item{\rm(v)} Maz'ya's iso-capacitary inequality
$$
V(\Omega)^{\frac{n-2}{n}}\lesssim cap_2(\Omega)
$$
holds for all pre-compact open sets $\Omega\subseteq M^n$.

\item{\rm(vi)} On-diagonal upper bound of heat kernel
$$
\sup_{x\in M^n}p_t(x,x)\lesssim t^{-\frac{n}{2}}
$$
holds for all $t>0$.

\item{\rm(vii)} Off-diagonal upper bound of heat kernel
$$
\sup_{(x,y)\in M^n\times M^n}p_t(x,y)\lesssim t^{-\frac{n}{2}}
$$
holds for all $t>0$.

\item{\rm(viii)} Lorentz-Green's global inequality
$$
\sup_{x\in M^n}V\big(\{y\in M^n:\ G(x,y)\ge t\}\big)\lesssim
t^{-\frac{n}{n-2}}
$$
holds for all $t>0$, with $G(x,y)$ being finite for $y\not=x$ (i.e.,
$M^n$ being non-parabolic).

\end{theorem}

\begin{proof} A simple proof of the equivalence (i)$\Leftrightarrow$(ii) can be found in \cite{BaCoLeSa} by Bakry, Coulhon, Ledoux and
Saloff-Coste; (i)$\Leftrightarrow$(vi) is due to Varopoulos
\cite{Va}; (ii)$\Leftrightarrow$(vi) is due to
Carlen-Kusuoka-Stroock \cite{CaKuSt}; (iii)$\Leftrightarrow$(vi) is
due to Davies \cite{Da} and Bakry \cite{Ba};
(iv)$\Leftrightarrow$(i)/(vi) is due to Carron \cite{Ca} and
Grigor'yan \cite{Gr} (cf. \cite{Gr1}); (v)$\Leftrightarrow$(i) is
due to Maz'ya \cite{Ma}; (vi)$\Leftrightarrow$(i) is due to Nash
\cite{Na}; (vii)$\Leftrightarrow$(i)/(ii)/(iii) are shown in
Varopoulos \cite{Va}, Carlen-Kusuoka-Stroock \cite{CaKuSt} and Bakry
\cite{Ba}; (viii)$\Leftrightarrow$(iv) is due to Carron \cite{Ca}.
\end{proof}

As a by-product of the foregoing theorem, we find two more new
conditions deciding geometry of $(M^n,ds^2)$.

\begin{theorem}\label{thm1+} Let $(M^n,ds^2)$, $n\ge 3$, be non-compact complete
boundary-free Riemannian manifold. Then anyone from (i) to (viii) in
Theorem \ref{thm1} is equivalent to each of the following two
statements with coarse constants:

\item{\rm(ix)} Lorentz-Green's local inequality
$$
\sup_{x\in\Omega}V\big(\{y\in\Omega:\ G_\Omega(x,y)\ge
t\}\big)\lesssim\big(t+V(\Omega)^\frac{2-n}{n}\big)^\frac{n}{2-n}
$$
holds for all regular bounded open sets $\Omega\subseteq M^n$ and
every $t>0$.

\item{\rm(x)} Green's potential-volume inequality
$$
\sup_{x\in \Omega}\int_{\Omega} G_\Omega(x,y)\,dV\lesssim
V(\Omega)^\frac2n
$$
holds for all regular bounded open sets $\Omega\subseteq M^n$.
\end{theorem}

\begin{proof} Since (iv)$\Leftrightarrow$(viii) above is valid, it
suffices to check
(viii)$\Rightarrow$(ix)$\Rightarrow$(x)$\Rightarrow$(iv). Suppose
(viii) is true. Note that for any regular bounded open set
$\Omega\subseteq M^n$ with $V(\Omega)<\infty$ we have
$$
y\in\Omega\quad\&\quad G_\Omega(x,y)\ge t\Longrightarrow
y\in\Omega\quad\&\quad G(x,y)\ge t,
$$
whence reaching (ix) via
$$
V\big(\{y\in \Omega:\ G_\Omega(x,y)\ge
t\}\big)\lesssim\min\{V(\Omega), t^{-\frac{n}{n-2}}\}\approx
\big(V(\Omega)^\frac{2-n}{n}+t\big)^\frac{n}{2-n}.
$$
If (ix) is valid, then an integration by parts and a change of
variables yield
\begin{eqnarray*}
\int_{\Omega}G_\Omega(x,y)\,dV(y)&=&
\int_0^\infty V\big(\{y\in\Omega:\ G_\Omega(x,y)\ge t\}\big)\,dt\\
&\lesssim&\int_0^\infty \big(V(\Omega)^\frac{2-n}{n}+t\big)^\frac{n}{2-n}\,dt\\
&\lesssim&\int_0^\infty
t\big(V(\Omega)^\frac{2-n}{n}+t\big)^\frac{2(n-1)}{2-n}\,dt\\
&\approx& V(\Omega)^\frac{2(n-1)}{n}\int_0^\infty
t\big(1+V(\Omega)^\frac{n-2}{n}t\big)^\frac{2(n-1)}{2-n}\,dt\\
&\approx& V(\Omega)^\frac{2}{n}\int_0^\infty
s(1+s)^\frac{2(n-1)}{2-n}\,ds,
\end{eqnarray*}
namely, (x) holds. Now, suppose (x) is valid. Given a regular
bounded open set $\Omega\subseteq M^n$ with $V(\Omega)<\infty$ (the
condition (iv) is trivially valid for $V(\Omega)=\infty$). Assume
that $u\not\equiv 0$ solves
\[
\left\{\begin{array} {r@{\;,\quad}l}
\Delta u(y)=\lambda_1(\Omega)u(y) & y\in\Omega,\\
u(y)=0 & y\in \partial\Omega.
\end{array}
\right.
\]
Then for each $x\in\Omega$ we have
\begin{eqnarray*}
u(x)&=&\int_\Omega G_\Omega(x,y)\Delta u(y)\,dV(y)\\
&=&\lambda_1(\Omega)\int_\Omega G_\Omega(x,y)u(y)\,dV(y)\\
&\le&\lambda_1(\Omega)\big(\sup_{y\in\Omega}u(y)\big)\int_\Omega G_\Omega(x,y)\,dV(y),\\
\end{eqnarray*}
thereby deriving
$$
1\le\lambda_1(\Omega)\sup_{x\in\Omega}\int_\Omega
G_\Omega(x,y)\,dV(y)\lesssim\lambda_1(\Omega)V(\Omega)^\frac{2}{n}.
$$
In other words, (iv) is true.
\end{proof}

Here it is worth observing that if $(M^n,ds^2)$ allows the
coarse/generic isoperimetric inequality above to exist then (vi),
(vii) and (viii) hold -- see also \cite{Ch}, \cite{ChFe1},
\cite{ChFe2}, and \cite{Gr2}, and hence the remaining conditions in
Theorems \ref{thm1} \& \ref{thm1+} hold. Evidently, the
previously-described conditions (iv)-(v)-(vi)-(vii)-(viii)-(ix)-(x)
may be regarded as seven different but equivalent geometric criteria
(involving no curvature hypotheses) for three different but
equivalent inequalities (i)-(ii)-(iii) above to hold. Since the
generic $L^1$ Sobolev inequality is strictly stronger than the
generic $L^2$ Sobolev inequality, the isoperimetric inequality
without best constant implies all the seven geometric inequalities
but not conversely in general -- nevertheless the implication can be
reversed when $(M^n,ds^2)$ has a nonnegative Ricci curvature -- see
also \cite[Lemma 8.1 \& Theorem 8.4]{Heb} as well as \cite{Ca},
\cite{CoLe}, \cite{Va}. Here it is also worth noticing that from Yau
\cite{Yau}, Varopoulos \cite{Va} and J. Li \cite{LiJ} we see that if
$(M^n,ds^2)$ has non-negative Ricci curvature then the generic $L^p$
($1\le p<n$) Sobolev inequality is true on $M^n$ when and only when
$V\big(B_r(x)\big)\gtrsim r^n$. Furthermore, Theorem \ref{thm1+} has
actually motivated us to establish the following rough comparison
principle for Green's function integrals which is of independent
interest.

\begin{theorem}\label{cor1} Let $(M^n,ds^2)$, $n\ge 3$, be a non-compact
complete boundary-free Riemannian manifold with anyone of Theorem
\ref{thm1} (i) through Theorem \ref{thm1+} (x) holding. If $0\le
q<p<\frac{n}{n-2}$ then the following comparison inequality with
coarse constant
$$
\left(\int_{\Omega}G_\Omega(x,y)^p\,dV(y)\right)^\frac{n}{n-p(n-2)}\lesssim
\left(\int_\Omega G_\Omega(x,y)^q\,dV(y)\right)^\frac{n}{n-q(n-2)}
$$
holds for all regular bounded open sets $\Omega\subseteq M^n$
containing $x$. Moreover, $G_\Omega(x,\cdot)$ belongs uniformly to
the Lorentz space $L^\ast_{\frac{n}{n-2}}(\Omega)$, namely,
$$
\sup_{(t,x)\in
(0,\infty)\times\Omega}t^\frac{n}{n-2}V\big(\{y\in\Omega:\
G_\Omega(x,y)\ge t\}\big)<\infty.
$$
\end{theorem}
\begin{proof} Due to the conditions from (i) to (x) in Theorems \ref{thm1} and
\ref{thm1+} are all equivalent, we may use any of them in what
follows.

Suppose $\int_\Omega G_\Omega(x,y)^q\,dV(y)$ is finite -- otherwise
there is nothing to argue.

Case 1: $q=0$. This ensures $V(\Omega)<\infty$. Using Theorem
\ref{thm1+} (ix), we obtain
\begin{eqnarray*}
\int_\Omega G_\Omega(x,y)^p\,dV(y)&=&\int_0^\infty
V\big(\{y\in\Omega:\ G_\Omega(x,y)\ge t\}\big)\,dt^p\\
&\lesssim&\int_0^\infty\big(t+V(\Omega)^\frac{2-n}{n}\big)^\frac{n}{2-n}\,dt^p\\
&\lesssim& V(\Omega)^\frac{n-p(n-2)}{n}\int_0^\infty
s^{p-1}(1+s)^\frac{n}{2-n}\,ds,
\end{eqnarray*}
as desired.

Case 2: $q>0$. Note that
$$
\int_{\Omega}G_\Omega(x,y)^p\,dV(y)=\frac{p}{q}\int_0^\infty
\frac{\Big(t^pV\big(\{y\in\Omega:\ G_\Omega(x,y)\ge
t\}\big)\Big)^\frac{p-q}{p}}{\Big(V\big(\{y\in\Omega:\
G_\Omega(x,y)\ge t\}\big)\Big)^{-\frac{q}{p}}}\,dt^q
$$
and for $\alpha=p,q$ and $r>0$,
$$
\int_0^\infty V\big(\{y\in\Omega:\ G_\Omega(x,y)\ge t\}\big)\,
dt^\alpha\ge r^\alpha V\big(\{y\in\Omega:\ G_\Omega(x,y)\ge
r\}\big).
$$
So
$$
\int_{\Omega}G_\Omega(x,y)^p\,dV(y)\le\frac{p}{q}\int_0^\infty
\frac{\Big(V\big(\{y\in\Omega:\ G_\Omega(x,y)\ge
r\}\big)\Big)^\frac{q}{p}}{\Big(\int_{\Omega}G_\Omega(x,y)^p\,dV(y)\Big)^\frac{q-p}{p}}\,dr^q.
$$
This yields
$$
\left(\int_{\Omega}G_\Omega(x,y)^p\,dV(y)\right)^\frac{q}{p}\le\frac{p}{q}\int_0^\infty\Big(V\big(\{y\in\Omega:\
G_\Omega(x,y)\ge r\}\big)\Big)^\frac{q}{p}\,dr^q.
$$
Furthermore, using Theorem \ref{thm1+} (ix) again we obtain
\begin{eqnarray*}
\int_0^\infty\Big(V\big(\{y\in\Omega:\ G_\Omega(x,y)\ge
r\}\big)\Big)^\frac{q}{p}\,dr^q&\lesssim&\int_0^t\Big(r^{-q}\int_\Omega
G_\Omega(x,y)^q\,dV(y)\Big)^\frac{q}{p}\,dr^q\\
&&+\int_t^\infty
r^\frac{qn}{p(2-n)}\,dr^q\\
&\approx&\left(\int_\Omega
G_\Omega(x,y)^q\,dV(y)\right)^\frac{q}{p}t^{q(1-\frac{q}{p})}\\
&&+\ t^{\frac{q(n-p(n-2))}{p(2-n)}}.
\end{eqnarray*}
If
$$
t=\left(\int_\Omega
G_\Omega(x,y)^q\,dV(y)\right)^\frac{n-2}{q(n-2)-n}
$$
then under $0<q<p<\frac{n}{n-2}$,
$$
\int_0^\infty\Big(V\big(\{y\in\Omega:\ G_\Omega(x,y)\ge
r\}\big)\Big)^\frac{q}{p}\,dr^q\lesssim\left(\int_\Omega
G_\Omega(x,y)^q\,dV(y)\right)^{\frac{1-\frac{n}{p(n-2)}}{1-\frac{n}{q(n-2)}}}
$$
and hence
$$
\int_\Omega G_\Omega(x,y)^p\,dV(y)\lesssim\left(\int_\Omega
G_\Omega(x,y)^q\,dV(y)\right)^{\frac{p-\frac{n}{n-2}}{q-\frac{n}{n-2}}}.
$$

Clearly, the second fact stated in Theorem \ref{cor1} follows
immediately from Theorem \ref{thm1+} (ix).
\end{proof}

\begin{remark}\label{rem1} Under the same hypothesis as in Theorem \ref{cor1}, we have that if $\frac{n}{2-n}<p<0$ then $0<-p<\frac{n}{n-2}$
and hence the Cauchy-Schwarz inequality and Theorem \ref{cor1} yield
\begin{eqnarray*}
V(\Omega)&\le&\Big(\int_\Omega
G_\Omega(x,y)^p\,dV(y)\Big)^\frac12\Big(\int_\Omega
G_\Omega(x,y)^{-p}\,dV(y)\Big)^\frac12\\
&\lesssim&\Big(\int_\Omega G_\Omega(x,y)^p\,dV(y)\Big)^\frac12
V(\Omega)^\frac{n+p(n-2)}{2n},
\end{eqnarray*}
namely,
$$
V(\Omega)^\frac{n-p(n-2)}{n}\lesssim\int_\Omega
G_\Omega(x,y)^p\,dV(y).
$$
Clearly, this last estimate can be regarded as a kind of the
isoperimetric inequality involving the Green function. Moreover, for
$0\le p<\frac{n}{n-2}$ one has the following area-volume-type
isoperimetric estimate:
$$
\left(\int_\Omega
G_\Omega(x,y)^p\,dV(y)\right)^\frac{n-1}{n-p(n-2)}\lesssim
S(\partial\Omega),
$$
upon the coarse isoperimetric inequality being valid for all regular
bounded domains $\Omega\subseteq M^n$ containing $x$.
\end{remark}

\section{Precise Estimates}

In order to find out the sharp versions of the estimates established
in Theorems \ref{thm1}-\ref{thm1+}-\ref{cor1}, we need to review the
celebrated result (due to Federer-Fleming \cite{FeFl} and Maz'ya
\cite{MaAd} for $M^n=\mathbb R^n$) that the sharp $L^1$-Sobolev
inequality
$$
\left(\int_{M^n}|f|^\frac{n}{n-1}\,dV\right)^\frac{n-1}{n}\le
(n\omega_n^\frac1n)^{-1}\int_{M^n}|\nabla f|\,dV\quad\hbox{for\ all\
functions}\ f\in C_0^\infty(M^n),
$$
is equivalent to the optimal isoperimetric inequality of Euclidean
type
$$
V(\Omega)^\frac{n-1}{n}\le
(n\omega_n^\frac1n)^{-1}S(\partial\Omega)\quad\hbox{for\ all\
smooth\ bounded\ domains}\ \Omega\subseteq M^n.
$$
Moreover, according to Hebey \cite[page 244]{Heb} or Ledoux
\cite{Led} we see that if $(M^n,ds^2)$ has nonnegative Ricci
curvature then the just-mentioned isoperimetric inequality is valid
only when $M^n$ is isometric to $\mathbb R^n$. So, when studying
sharp geometric forms of the $L^2$-Sobolev inequality, -Nash
inequality, and -logarithmic Sobolev inequality of Euclidean type,
we may naturally derive the following result which is partially
known.

\begin{theorem}\label{thm2} Let $(M^n,ds^2)$, $n\ge 3$, be a non-compact
complete boundary-free Riemannian manifold with the sharp
isoperimetric inequality of Euclidean type. Then the following
sharp:

\item{\rm(i)} $L^2$ Sobolev inequality
$$
\left(\int_{M^n}|f|^\frac{2n}{n-2}\,dV\right)^\frac{n-2}{n}\le\big(n(n-2)\big)^{-1}\left(\frac{\Gamma(n)}{\Gamma(\frac{n}2)\Gamma(1+\frac{n}2)\omega_n}\right)^\frac2n
\int_{M^n}|\nabla f|^2\,dV
$$
holds for all $f\in C^\infty_0(M^n)$, where $\Gamma(\cdot)$ is the
standard Gamma function.

\item{\rm(ii)} Nash's inequality
$$
\Big(\int_{M^n}|f|^2\,dV\Big)^{1+\frac{2}{n}}\le\frac{(n+2)^\frac{n+2}{n}}{n(2\omega_n)^\frac2n\lambda_N}\Big(\int_{M^n}|f|\,dV\Big)^\frac4n\int_{M^n}|\nabla
f|^2\,dV
$$
holds for all $f\in C^\infty_0(M^n)$, where $\lambda_N$ is the first
non-zero Neumann eigenvalue of $\Delta$ on radial functions on the
unit ball of $\mathbb R^n$.

\item{\rm(iii)} $L^2$ logarithmic Sobolev inequality
$$
\exp\Big(\frac2n\int_{M^n}|f|^2\log |f|^2\,dV\Big)\le\frac{2}{en\pi}\int_{M^n}|\nabla f|^2\,dV
$$
holds for all $f\in C^\infty_0(M^n)$ with $\int_{M^n}|f|^2\,dV=1$.

\item{\rm(iv)} Faber-Krahn's eigenvalue inequality
$$
\inf_{D\subseteq\mathbb
R^n}\lambda_{1,e}(D)V_e(D)^\frac2n\le{\lambda_1(\Omega)} V(\Omega)^{\frac2n}
$$
holds for all regular bounded open sets $\Omega\subseteq M^n$, where
the infimum ranges over all regular bounded open sets
$D\subseteq\mathbb R^n$ and $\lambda_{1,e}(D)$ and $V_e(D)$ are
respectively the eigenvalue and volume associated with $D$ under the
Euclidean metric.

\item{\rm(v)} Maz'ya's iso-capacitary inequality
$$
V(\Omega)^{\frac{n-2}{n}}\le\big(n(n-2)\omega_n^\frac2n\big)^{-1}
cap_2(\Omega)
$$
holds for all pre-compact open sets $\Omega\subseteq M^n$.

\item{\rm(vi)} On-diagonal bound of heat kernel
$$
\sup_{x\in M^n}p_t(x,x)\le(4\pi t)^{-\frac{n}{2}}
$$
holds for all $t>0$.

\item{\rm(vii)} Off-diagonal bound of heat kernel
$$
\sup_{(x,y)\in M^n\times M^n}p_t(x,y)\le(4\pi t)^{-\frac{n}{2}}
$$
holds for all $t>0$.

\item{\rm(viii)} Lorentz-Green's global inequality
$$
\sup_{x\in M^n}V\big(\{y\in M^n:\ G(x,y)\ge
t\}\big)\le\big(n(n-2)\omega_n^\frac2n\big)^{-\frac{n}{n-2}}
t^{-\frac{n}{n-2}}
$$
holds for all $t>0$, with $G(x,y)$ being finite for $y\not=x$ (i.e.,
$M^n$ being non-parabolic).

On the other hand, if anyone of (i), (ii), (iii), (v), (vi), (vii),
((iv), (viii)) is true and $(M^n,ds^2)$ has nonnegative Ricci
curvature then $M^n$ is isometric (diffeomorphic) to $\mathbb R^n$.
\end{theorem}

\begin{proof} To begin with, the sharpness of the above eight statements is due to
the fact that all equalities there can occur when $M^n=\mathbb R^n$.

Next, let us check each statement. Note that the Euclidean
counterpart of (v) is Maz'ya's sharp isocapacitary estimate in
\cite[page 105, (7)]{Ma0} whose proof depends only on the sharp
Euclidean isoperimetric inequality. So, (v) is true under the
hypothesis. The validity of (v) is used to imply
\begin{eqnarray*}
\int_{M^n}|f|^\frac{2n}{n-2}\,dV&=&\int_0^\infty V\big(\{x\in
M^n:\
|f(x)|\ge t\}\big)\,dt^\frac{2n}{n-2}\\
&\le&\frac{\int_0^\infty
cap_2\big(\{x\in M^n:\ |f(x)|\ge
t\}\big)^\frac{n}{n-2}\,dt^\frac{2n}{n-2}}{\Big((n(n-2))^\frac12\omega_n^\frac1n\Big)^\frac{2n}{n-2}}.
\end{eqnarray*}
According to Maz'ya's \cite[Remark 5 \& Proposition 1]{Ma}, we have
$$
\left(\int_0^\infty cap_2\big(\{x\in M^n:\ |f(x)|\ge
t\}\big)^\frac{n}{n-2}\,dt^\frac{2n}{n-2}\right)^\frac{n-2}{2n}\\
\le\frac{\left(\int_{M^n}|\nabla f|^2\,dV\right)^\frac12}{\Big(\frac{\Gamma(\frac{n}{2})\Gamma(1+\frac{n}{2})}{\Gamma(n)}\Big)^\frac1n},
$$
therefore deriving (i).

According to Ni's argument for Perelman's proposition (cf.
\cite[Proposition 4.1]{Ni}) we find easily that if $(M^n,ds^2)$
allows the optimal isoperimetric inequality of Euclidean type then
(iii) holds. This in turn implies (vii) and so (vi) -- see
Bakry-Concordet-Ledoux \cite[Theorem 1.2]{BaCoLe}.

The verification of (ii) is more or less contained in
Druet-Hebey-Vaugon's argument for \cite[Theorem 5.1]{DrHeVau}. To
see this, we may just prove that (ii) is valid for any function
$f\in C_0^\infty(M^n)$ which is continuous and has only
non-degenerate critical points in its support. As with such a
function $f$, let $g:\mathbb R^n\mapsto\mathbb R$ be nonnegative,
radial, decreasing with respect to $|x|$, and be determined by
$$
V_e\big(\{x\in\mathbb R^n:\ g(x)\ge t\}\big)=V\big(\{x\in M^n:\
f(x)\ge t\}\big).
$$
Then $g$ has a compact support in $\mathbb R^n$ and enjoys
$$
\int_{M^n}f^j\,dV=\int_{\mathbb R^n}g^j\,dV_e,\quad j=1,2
$$
and
\begin{eqnarray*}
-\int_{f^{-1}[t]}|\nabla f|^{-1}\,dS&=&\frac{d}{dt}V\big(\{x\in
M^n:\ f(x)\ge t\}\big)\\
&=&\frac{d}{dt}V_e\big(\{x\in\mathbb R^n:\ g(x)\ge
t\}\big)\\
&=&-\int_{g^{-1}[t]}|\nabla_e g|^{-1}\,dS_e,
\end{eqnarray*}
where $f^{-1}[t]$ and $g^{-1}[t]$ are the pre-images of $t$ under
$f$ and $g$ respectively; $dS$ and $dS_e$ denote the area elements
associated with $M^n$ and $\mathbb R^n$ respectively; and $\nabla_e$
stands for the Euclidean gradient.

Since
$$
(n\omega_n)^{-1}=\frac{V_e(U)^\frac{n-2}{n}}{S_e(\partial U)}
$$
holds for any Euclidean ball $U\subseteq\mathbb R^n$, the sharp
isoperimetric inequality of Euclidean type is applied to yield
$$
S_e\big(g^{-1}[t]\big)\le S\big(f^{-1}[t]\big),\quad t>0.
$$
Note that $|\nabla_e g|$ equals a positive constant on $g^{-1}[t]$.
So the last inequality plus the Cauchy-Schwarz inequality implies
\begin{eqnarray*}
\left(\int_{g^{-1}[t]}|\nabla_e
g|\,dS_e\right)\left(\int_{g^{-1}[t]}\frac{dS_e}{|\nabla_e
g|}\right)&=&S_e\big(g^{-1}[t]\big)^2\le
S\big(f^{-1}[t]\big)^2\\
&\le&\left(\int_{f^{-1}[t]}|\nabla
f|\,dS\right)\left(\int_{f^{-1}[t]}\frac{dS}{|\nabla f|}\right),
\end{eqnarray*}
and consequently,
$$
\int_{g^{-1}[t]}|\nabla_e g|\,dS_e\le\int_{f^{-1}[t]}|\nabla
f|\,dS,\quad t>0.
$$
Now the co-area formula, along with the last inequality, gives
\begin{eqnarray*}
\int_{\mathbb R^n} |\nabla_e
g|^2\,dV_e&=&\int_0^\infty\left(\int_{g^{-1}[t]}|\nabla_e
g|\,dS_e\right)dt\\
&\le&\int_0^\infty\left(\int_{f^{-1}[t]}|\nabla
f|\,dS\right)dt\\
&=&\int_{\Omega}|\nabla f|^2\,dV.
\end{eqnarray*}
Now an application of the sharp Nash's inequality on $\mathbb R^n$
(due to Carlen-Loss \cite{CaLo}) produces
\begin{eqnarray*}
\left(\int_{M^n}f^2\,dV\right)^{1+\frac2n}&=&\left(\int_{\mathbb
R^n}g^2\,dV_e\right)^{1+\frac2n}\\
&\le&\frac{(n+2)^\frac{n+2}{n}}{n(2\omega_n)^\frac2n\lambda_N}\Big(\int_{\mathbb
R^n}g\,dV_e\Big)^\frac4n\int_{\mathbb R^n}|\nabla_e g|^2\,dV_e\\
&\le&\frac{(n+2)^\frac{n+2}{n}}{n(2\omega_n)^\frac2n\lambda_N}\Big(\int_{M^n}
f\,dV\Big)^\frac4n\int_{M^n}|\nabla f|^2\,dV,
\end{eqnarray*}
as desired.

To reach (iv), let $\Omega$ be any regular bounded open subset of
$M^n$ and $f\in C_0^\infty(\Omega)$ be nonnegative. In a similar
manner to proving the optimal Nash's inequality above, we may choose
a Euclidean ball $B\subseteq \mathbb R^n$ such that
$V_e(B)=V(\Omega)$. Given $t>0$, let $\Omega_t=\{x\in\Omega:\
f(x)\ge t\}$ and $g$ be a nonnegative radial $C^\infty_0(B)$
function with
$$
B_t=\{y\in B:\ g(y)\ge t\};\quad  V_e(B_t)=V\big(\Omega_t\big).
$$
Then
$$
-\int_{\partial B_t}|\nabla_e
g|^{-1}\,dS_e=\frac{dV_e(B_t)}{dt}=\frac{dV(\Omega_t)}{dt}=-\int_{\partial
\Omega_t}|\nabla f|^{-1}\,dS.
$$
Additionally,
$$
\int_\Omega f^2\,dV=\int_0^\infty V(\Omega_t)\,dt^2=\int_0^\infty
V_e(B_t)\,dt^2=\int_{B}g^2\,dV_e.
$$
From the sharp isoperimetric inequality of Euclidean type, the fact
that $|\nabla_e g|$ is equal to a non-zero constant on $\partial
B_t$, and the Cauchy-Schwarz inequality, we conclude
\begin{eqnarray*}
\left(\int_{\partial B_t}|\nabla_e g|\,dS_e\right)\left(\int_{\partial B_t}|\nabla_e g|^{-1}\,dS_e\right)&=&S_e(\partial B_t)^2\le S(\partial\Omega_t)^2\\
&\le&\left(\int_{\partial \Omega_t}|\nabla f|\,dS\right)\left(\int_{\partial\Omega_t}|\nabla f|^{-1}\,dS\right),
\end{eqnarray*}
and so,
$$
\int_{\partial B_t}|\nabla_e g|\,dS_e\le\int_{\partial
\Omega_t}|\nabla f|\,dS,\quad t>0.
$$
Furthermore, the co-area formula is used once again to deduce
$$
\int_B |\nabla_e g|^2\,dV_e=\int_0^\infty\left(\int_{\partial B_t}|\nabla_e g|\,dS_e\right)dt\le\int_0^\infty\left(\int_{\partial\Omega_t}|\nabla f|\,dS\right)dt=\int_{\Omega}|\nabla f|^2\,dV.
$$
As a result, we find
$$
\frac{\int_\Omega f^2\,dV}{\int_\Omega |\nabla
f|^2\,dV}\le\frac{\int_B g^2\,dV_e}{\int_B |\nabla_e g|^2\,dV_e},
$$
thereby getting
$$
\lambda(\Omega)^{-1}\le\frac{V(\Omega)^\frac2n}{\lambda_{1,e}(B)V_e(B)^\frac2n}\le\frac{V(\Omega)^\frac2n}{\inf_{D\subseteq\mathbb
R^n}\lambda_{1,e}(D)V_e(D)^\frac2n}.
$$

The statement (viii) follows from the forthcoming Theorem
\ref{thm2+} (ix) since there is a sequence
$\{\Omega_i\}_{i=1}^\infty$ of regular bounded open subsets of $M^n$
such that
$$
x\in\Omega_1\subseteq\Omega_2\subseteq\cdots\subseteq M^n,\quad
\cup_{i=1}^\infty\Omega_i=M^n\quad\hbox{and}\quad
G_{\Omega_i}(x,\cdot)\nearrow G(x,\cdot)
$$
for a given point $x\in M^n$; see also \cite{LiT}.

Finally, let us handle the rest of Theorem \ref{thm2}.

It is known that if (i)/(ii)/(iii)/(vi)/(vii) is true and
$(M^n,ds^2)$ has nonnegative Ricci curvature then $M^n$ is isometric
to $\mathbb R^n$ -- see Varopoulos \cite{Va} and Ledoux \cite{Led},
Druet-Hebey-Vaugon \cite{DrHeVau}, Xia \cite{Xia1} (cf. Xia's paper
\cite{Xia2} regarding Galiardo-Nirenberg's inequalities on manifolds
of nonnegative Ricci curvature and Lutwark-Yang-Zhang's work
\cite{LuYaZh} on the optimal affine Galiardo-Nirenberg inequality of
Euclidean type), Bakry-Concordet-Ledoux \cite{BaCoLe}, Ni \cite{Ni}
(cf. Ni \cite{Ni1} and Kotschwar-Ni \cite{KoNi} extending the sharp
$L^p (1<p<n)$ Sobolev logarithmic inequality in DelPino-Dolbeault
\cite{DPiDo}), and Li \cite{LiP}. Since (v) implies (i), if (v) is
valid with $(M^n,ds^2)$ having nonnegative Ricci curvature then
$M^n$ must be isometric to $\mathbb R^n$.

In accordance with Theorem \ref{thm1}, if anyone of (iv) and (viii)
is valid then Theorem \ref{thm1} (vii) is true and hence from Li-Yau
\cite{LiYau} it follows that when $B_r(x)$ represents the geodesic
ball $\{y\in M^n: d(y,x)<r\}$ with radius $r>0$ and center $x\in
M^n$ we get
$$
1\lesssim\liminf_{t\to\infty}
V\big(B_t(x)\big)p_{t^2}(x,y)\quad\hbox{and}\quad
1\lesssim\liminf_{t\to\infty} \frac{V\big(B_t(x)\big)}{\omega_n t^n}
\quad x,y\in M^n.
$$
Since the Ricci curvature of $(M^n,ds^2)$ is nonnegative, an
application of Gromov's comparison theorem (cf. \cite[page 11]{Heb})
produces a constant $\kappa$ depending only on $n$ such that
$$
\kappa\le\frac{V\big(B_r(x)\big)}{\omega_n r^n}\le 1\quad\hbox{for\
all}\ r>0\ \hbox{and}\ x\in M^n.
$$
This last estimate indicates that $\kappa\le 1$ is true.
Furthermore, from Cheeger-Colding \cite{ChCo} it follows that $M^n$
is diffeomorphic to $\mathbb R^n$.
\end{proof}

The follow-up seems quite natural.

\begin{theorem}\label{thm2+} Let $(M^n,ds^2)$, $n\ge 3$, be a non-compact
complete boundary-free Riemannian manifold with the sharp
isoperimetric inequality of Euclidean type. Then the following
sharp:

\item{\rm(ix)} Lorentz-Green's local inequality
$$
\sup_{x\in\Omega}V\big(\{y\in\Omega:\ G_\Omega(x,y)\ge
t\}\big)\le\big(n(n-2)\omega_n^\frac2n
t+V(\Omega)^\frac{2-n}{n}\big)^\frac{n}{2-n}
$$
holds for all regular bounded open sets $\Omega\subseteq M^n$ and
every $t>0$.

\item{\rm(x)} Green's potential-volume inequality
$$
\sup_{x\in \Omega}\int_{\Omega} G_\Omega(x,y)\,dV(y)\le
(2n\omega_n^\frac2n)^{-1}V(\Omega)^\frac2n
$$
holds for all regular bounded open sets $\Omega\subseteq M^n$.

On the other hand, if either (ix) or (x) is true under $(M^n,ds^2)$
being of nonnegative Ricci curvature then $M^n$ is diffeomorphic to
$\mathbb R^n$.
\end{theorem}

\begin{proof} The sharpness of (ix) and (x) can be verified by taking $M^n=\mathbb R^n$ and $\Omega=B_r(x)$.

To prove (ix), for the sake of simplicity let us write
$$
V_t=V\big(\{y\in\Omega:\ G_\Omega(x,y)\ge t\}\big)\quad\hbox{for\ \
given}\ \ x\in\Omega\ \ \hbox{and\ \ any}\ \ t\ge 0.
$$
Note that
$$
\frac{d V_t}{dt}=-\int_{\{y\in\Omega:\ G_\Omega(x,y)=t\}}|\nabla
G_\Omega(x,y)|^{-1}\,dS
$$
and
$$
1=\int_{\{y\in\Omega:\ G_\Omega(x,y)=t\}}|\nabla G_\Omega(x,y)|\,dS.
$$
So, a combined application of the Cauchy-Schwarz inequality and the
sharp isoperimetric inequality of Euclidean type gives
$$
-\frac{d V_t}{dt}\ge\Big(\int_{\{y\in\Omega:\
G_\Omega(x,y)=t\}}\,dS\Big)^2\ge\Big((n\omega_n^\frac1n)V_t^\frac{n-1}{n}\Big)^2.
$$
Integrating both sides of the last differential inequality over
$[t_1,t_2]$ where $0\le t_1<t_2<\infty$, we obtain
$$
V_{t_2}^\frac{2-n}{n}\ge
V_{t_1}^\frac{2-n}{n}+n(n-2)\omega_n^\frac2n (t_2-t_1),
$$
whence deducing (ix).

To check (x), we apply the foregoing notations and the
just-demonstrated (ix) to achieve
\begin{eqnarray*}
\int_\Omega G_\Omega(x,y)\,dV(y)&=&\int_0^\infty V_t\,dt\\
&\le& n^2\omega_n^\frac2n V(\Omega)^\frac{2(n-1)}{n}\int_0^\infty\Big(1+n(n-2)V(\Omega)^\frac{n-2}{n}\omega_n^\frac2n t\Big)^\frac{2(n-1)}{n-2}\,dt\\
&=&(2n\omega_n^\frac2n)^{-1}V(\Omega)^\frac2n.
\end{eqnarray*}

According to Theorem \ref{thm1+}, if (iv) or (viii) is true then the
condition of Theorem \ref{thm1} (vii) is valid, and consequently the
argument for the rigidity part of Theorem \ref{thm2} can be used to
derive that $M^n$ is diffeomorphic to $\mathbb R^n$.
\end{proof}

In spirit of Theorem \ref{thm2+} we fortunately discover an optimal
version of Theorem \ref{cor1}.

\begin{theorem}\label{con2} Let $(M^n,ds^2)$, $n\ge 3$, be a non-compact
complete boundary-free Riemannian manifold with the sharp
isoperimetric inequality of Euclidean type. If $0\le
q<p<\frac{n}{n-2}$ then the following comparison inequality
$$
\left(\int_{\Omega}
G_\Omega(x,y)^p\,dV(y)\right)^\frac{n}{n-p(n-2)}\le\kappa_{n,p,q}
\left(\int_\Omega G_\Omega(x,y)^q\,dV(y)\right)^\frac{n}{n-q(n-2)}
$$
holds for any regular bounded open set $\Omega\subseteq M^n$
containing $x$, with equality when $(M^n,\Omega)=(\mathbb
R^n,B_r(x))$, where
$$
\kappa_{n,p,q}=\frac{\left(\frac{\big(n(n-2)\omega_n^\frac{2}{n}\big)^q}{qB\big(\frac{n}{n-2}-q,q\big)}\right)^\frac{n}{n-q(n-2)}
}{\left(\frac{\big(n(n-2)\omega_n^\frac{2}{n}\big)^p}{pB\big(\frac{n}{n-2}-p,p\big)}\right)^\frac{n}{n-p(n-2)}}.
$$
Moreover, if
$$
\mathsf{G}(p;
x,\Omega)=\left(\frac{\big(n(n-2)\omega_n^\frac2n\big)^p}{pB\big(\frac{n}{n-2}-p,p\big)}\int_\Omega
G_\Omega(x,y)^p\,dV(y)\right)^\frac{n}{n-p(n-2)}
$$
then
$$
\lim_{p\to\frac{n}{n-2}}\mathsf{G}(p;x,\Omega)=\lim_{t\to\infty}D_\Omega(x,t)^\frac{n}{2-n},
$$
where
$$
D_\Omega(x,t)=V\big(\{y\in\Omega: G_\Omega(x,y)\ge
t\}\big)^\frac{2-n}{n}-n(n-2)\omega_n^\frac2n t
$$
is referred to as the Lorentz norm deficit of $G_\Omega(x,\cdot)$ at
$(x,t)$. In particular, if $(M^n,ds^2)$ has nonnegative Ricci
curvature then
$$
\lim_{t\to\infty}D_\Omega(x,t)^\frac{n}{2-n}=\omega_nR_\Omega(x)^n,
$$
where
$$
R_\Omega(x)=\lim_{y\to x}\Big(d(x,y)^{2-n}-n(n-2)\omega_n
G_\Omega(x,y)\Big)^\frac{1}{2-n}
$$
is called the harmonic radius of $\Omega$ at $x$.
\end{theorem}
\begin{proof} Let us still use the notations introduced in the proofs of Theorems \ref{thm2} \& \ref{thm2+}. Without loss of generality we may assume $\int_\Omega G_\Omega(x,y)^q\,dV(y)<\infty$ -- otherwise there is nothing to argue.
Under the sharp isoperimetric inequality of Euclidean type we have
the following monotone inequality:
$$
D_\Omega(x,t)\ge D_\Omega(x,r)\quad\hbox{for}\quad t\ge r\ge 0.
$$

Case 1: $q=0$. Clearly, $0<p<\frac{n}{n-2}$, the last inequality and
the layer-cake formula imply
\begin{eqnarray*}
\int_{\Omega} G_\Omega(x,y)^p\,dV(y)&\le& p\int_0^\infty
t^{p-1}\big(V_0^\frac{2-n}{n}+n(n-2)\omega_n^\frac2n
t\big)^\frac{n}{2-n}\,dt\\
&=& V(\Omega)\int_0^\infty
t^{p-1}\big(1+V(\Omega)^\frac{n-2}{n}n(n-2)\omega_n^\frac2n
t\big)^\frac{n}{2-n}\,dt\\
&=&\frac{pV(\Omega)^{1-\frac{p(n-2)}{n}}}{\big(n(n-2)\omega_n^\frac2n\big)^p}\int_0^\infty
r^{p-1}(1+r)^\frac{n}{2-n}\,dr,
\end{eqnarray*}
as desired.

Case 2: $q>0$. For simplicity, set
$$
U_q(r)=-\int_r^\infty t^q dV_t.
$$
Through integrating by parts, changing variables and using the
previous monotone inequality we obtain
\begin{eqnarray*}
U_q(r)&\le& r^qV_r+q\int_r^\infty \big(D_\Omega(x,r)+n(n-2)\omega_n^\frac2n t\big)^\frac{n}{2-n} t^{q-1}\,dt\\
&=&n^2\omega_n^\frac2n\int_r^\infty\big(D_\Omega(x,r)+n(n-2)\omega_n^\frac2n t\big)^\frac{2(n-1)}{2-n} t^{q}\,dt\\
&=&\frac{n^2\omega_n^\frac2nD_\Omega(x,r)^{\frac{n-(n-2)q}{n-2}}}{\big(n(n-2)\omega_n^\frac2n\big)^{q+1}}\int_{\frac{n(n-2)\omega_n^\frac2n}{D_\Omega(x,r)}}^\infty(1+t)^\frac{2(n-1)}{2-n}
t^q\,dt,
\end{eqnarray*}
thereby getting
$$
\left(\frac{\big(n(n-2)\omega_n^\frac{2}{n}\big)^qU_q(r)}{qB\big(\frac{n}{n-2}-q,q\big)}\right)^\frac
n{n-q(n-2)}\le D_\Omega(x,r)^\frac{n}{2-n}.
$$
Note that
$$
\frac{dU_q(r)}{dr}=r^q \frac{dV_r}{dr}\le-(n\omega_n^\frac1n)^2 r^q
V_r^\frac{2(n-1)}{n}.
$$
So, the foregoing two inequalities produce the following
differential inequality
$$
r^q\left(\big(\alpha
U_q(r)\big)^\frac{2-n}{n-q(n-2)}+n(n-2)\omega_n^\frac2n
r\right)^\frac{2(n-1)}{2-n}\le -(n\omega_n^\frac1n)^{-2}\frac{d
U_q(r)}{dr},
$$
where
$$
\alpha=\frac{\big(n(n-2)\omega_n^\frac2n\big)^q}{qB\big(\frac{n}{n-2}-q,q\big)}.
$$
Since $U_q(r)\le U_q(0)$, the last differential inequality is used
to derive
$$
r^q\left(\alpha^\frac{2-n}{n-q(n-2)}+
U_q(0)^\frac{n-2}{n-q(n-2)}\Big(n(n-2)\omega_n^\frac2n
r\Big)\right)^\frac{2(n-1)}{2-n}\le
-(n\omega_n^\frac1n)^{-2}\frac{\frac{d
U_q(r)}{dr}}{U_q(r)^\frac{2(n-1)}{n-q(n-2)}}.
$$
Upon integrating this last inequality over $[0,s]$ against $dr$, we
get
\begin{eqnarray*}
&&U_q(s)^\frac{(2-n)(q+1)}{n-q(n-2)}\\
&&\le\ \frac{\int_0^s
\Big(\alpha^\frac{2-n}{n-q(n-2)}+rU_q(0)^\frac{n-2}{n-q(n-2)}\Big)^\frac{2(n-1)}{2-n}r^{q}{dr}}{
\Big(\frac{(n\omega_n^\frac1n)^2(n-2)(q+1)}{n-q(n-2)}\Big)^{-1}}+U_q(0)^\frac{(2-n)(q+1)}{n-q(n-2)}\\
&&=\
U_q(0)^\frac{(2-n)(q+1)}{n-q(n-2)}\left(1+\frac{\int_0^{\big(\alpha
U_q(0)\big)^\frac{n-2}{n-q(n-2)}n(n-2)\omega_n^\frac2n
s}(1+r)^\frac{2(n-1)}{2-n}r^q\,dr}{\Big(\frac{\alpha
n(q+1)}{\big(n(n-2)\omega_n^\frac2n\big)^q\big(n-q(n-2)\big)}\Big)^{-1}}\right).
\end{eqnarray*}
To shorten our notation, let
$$
\beta={\big(\alpha
U_q(0)\big)^\frac{n-2}{n-q(n-2)}n(n-2)\omega_n^\frac2n}.
$$
Then the last inequality, together with an integration-by-parts,
yields
\begin{eqnarray*}
U_p(0)&=&(p-q)\int_0^\infty U_q(s)s^{p-q-1}\,ds\\
&\le& U_q(0)\int_0^\infty\left(1+\frac{\int_0^{\beta
s}(1+r)^\frac{2(n-1)}{2-n}r^q\,dr}{\Big(\frac{\alpha
n(q+1)}{\big(n(n-2)\omega_n^\frac2n\big)^q\big(n-q(n-2)\big)}\Big)^{-1}}\right)^\frac{n-q(n-2)}{(2-n)(q+1)}
\,ds^{p-q}\\
&=&-U_q(0)\int_0^\infty
s^{p-q}\frac{d}{ds}\left(1+\frac{\int_0^{\beta
s}(1+r)^\frac{2(n-1)}{2-n}r^q\,dr}{\Big(\frac{\alpha
n(q+1)}{\big(n(n-2)\omega_n^\frac2n\big)^q\big(n-q(n-2)\big)}\Big)^{-1}}\right)^\frac{n-q(n-2)}{(2-n)(q+1)}\,ds\\
&=&\frac{\int_0^\infty
{u^p}{(1+u)^\frac{2(n-1)}{2-n}}\left(1+\frac{\int_0^u
v^q(1+v)^\frac{2(n-1)}{2-n}\,dv}{\Big(\frac{\alpha
n(q+1)}{\big(n(n-2)\omega_n^\frac2n\big)^q\big(n-(n-2)q\big)}\Big)^{-1}}\right)^\frac{2(n-1)}{(2-n)(q+1)}\,du}{\Big(\frac{\big(\alpha
U_q(0)\big)^{\frac{n-p(n-2)}{n-q(n-2)}}}{\big(\frac{n-2}{n}\big)\big(n(n-2)\omega_n^\frac2n\big)^{p}}\Big)^{-1}}\\
&\le&\frac{\big(\alpha
U_q(0)\big)^{\frac{n-p(n-2)}{n-q(n-2)}}}{\big(\frac{n-2}{n}\big)\big(n(n-2)\omega_n^\frac2n\big)^{p}}\int_0^\infty
u^p(1+u)^\frac{2(n-1)}{2-n}\,du\\
&=&\left(\frac{pB\big(\frac{n}{n-2}-p,p\big)}{\big(n(n-2)\omega_n^\frac2n\big)^{p}}\right)
\left(\Big(\frac{\big(n(n-2)\omega_n^\frac2n\big)^q}{qB\big(\frac{n}{n-2}-q,q\big)}\Big)
U_q(0)\right)^{\frac{n-p(n-2)}{n-q(n-2)}}.
\end{eqnarray*}
A simplification of the just-obtained equalities and inequalities,
along with
$$
U_p(0)=\int_\Omega G_\Omega(x,y)^p\,dV(y),
$$
gives the desired inequality.

Of course, if $M^n=\mathbb R^n$ and $\Omega=B_R(x)=\{y\in\mathbb
R^n:\ |x-y|<R\}$ (given $x\in\mathbb R^n$ and $R>0$), then

\[
G_\Omega(x,y)=\left\{\begin{array} {r@{\;,\quad}l}
\big(n(n-2)\omega_n\big)^{-1}\Big(|x-y|^{2-n}-\big(\frac{|y|}{R}\big)^{2-n}\big|\frac{R^2}{|y|^2}y-x\big|^{2-n}\Big)
& x\not=0,\\
\big(n(n-2)\omega_n\big)^{-1}\Big(|y|^{2-n}-R^{2-n}\Big) & x=0,
\end{array}
\right.
\]
and hence a direct computation yields the equality case of Theorem
\ref{con2} via
$$
\Big(\int_\Omega
G_\Omega(x,y)^p\,dV(y)\Big)^\frac{n}{n-p(n-2)}=\kappa_{n,p,0}V(\Omega).
$$
See also Weinberger \cite{We} or Bandle \cite[pages 59-61]{Ban}.

Last of all, let us deal with the limit formulas. Fixing $r\in
(0,\infty)$, employing
$$
{U_p(r)^\frac{n}{n-p(n-2)}}\le\kappa_{n,p,0}
D_\Omega(x,r)^\frac{n}{2-n},
$$
and the following elementary inequality (cf. \cite[page 389,
(17)]{Ada}):
$$
(a+b)^c\le a^c+c2^{c-1}(b^c+ba^{c-1})\quad \hbox{where}\quad a,b\ge
0;\ c\ge 1,
$$
we obtain that $U_p(r)\le \int_0^r V_t\,dt^p$ is valid for a
sufficient large $r$, whence reaching
\begin{eqnarray*}
\mathsf{G}(p;x,\Omega)
&\le&\left(\kappa_{n,p,0}^{\frac{p(n-2)-n}{n}}\Big(U_p(r)+\int_0^r V_t\,dt^p\Big)\right)^\frac{n}{n-p(n-2)}\\
&\le&\Big(V_r^\frac{2-n}{n}-n(n-2)\omega_n^\frac2n r\Big)^\frac{n}{2-n}+\left(\frac{\frac{n}{n-p(n-2)}}{\kappa_{n,p,0}2^{1-\frac{n}{n-p(n-2)}}}\right)\times\\
&&\times\left(\Big(\int_0^r V_t\,dt^p\Big)^\frac{n}{n-p(n-2)}+\frac{\int_0^r V_t\,dt^p}{U_p(r)^{1-\frac{n}{n-p(n-2)}}}\right)\\
&\le&D_\Omega(x,r)^\frac{n}{2-n}+\left(\frac{\frac{n}{n-p(n-2)}}{\kappa_{n,p,0}^\frac1n}\right)\Big(2\int_0^r
V_t\,dt^p\Big)^\frac{n}{n-p(n-2)}.
\end{eqnarray*}
Taking the asymptotic behavior of the following Beta function into
account:
$$
B\big(\frac{n}{n-2}-p,p+1\big)\sim\left(
\frac{\sqrt{2\pi}\big(\frac{n}{n-2}-p\big)^{\frac{n+2}{n-2}-p}(p+1)^{p+\frac12}}
{\big(\frac{2(n-1)}{n-2}\big)^\frac{3n-2}{n-2}}\right)^\frac{n}{p(n-2)-n},
$$
we gain
$$
\lim_{p\to\frac{n}{n-2}}\left(\frac{\frac{n}{n-p(n-2)}}{\kappa_{n,p,0}}\right)\Big(2\int_0^r
V_t\,dt^p\Big)^\frac{n}{n-p(n-2)}=0,
$$
and thus
$$
\lim_{p\to\frac{n}{n-2}}\mathsf{G}(p;x,\Omega)\le
D_\Omega(x,r)^\frac{n}{2-n}.
$$
Letting $r\to\infty$ on the right-hand-side of the last inequality,
we find
$$
\lim_{p\to\frac{n}{n-2}}\mathsf{G}(p;x,\Omega)\le\lim_{r\to\infty}D_\Omega(x,r)^\frac{n}{2-n}.
$$

At the same time, since $D_\Omega(x,s)^\frac{n}{2-n}$ decreases with
$s$, one has
\begin{eqnarray*}
U_p(0)&=&\int_0^\infty \Big(D_\Omega(x,s)+n(n-2)\omega_n^\frac2n
s\Big)^\frac{n}{2-n}\,ds^p\\
&\ge&\int_0^\infty\Big(\big(\lim_{s\to\infty}D_\Omega(x,s)\big)+n(n-2)\omega_n^\frac2n
s\Big)^\frac{n}{2-n}\,ds^p\\
&=&\lim_{s\to\infty}D_\Omega(x,s)^\frac{n-p(n-2)}{2-n}\big(n(n-2)\omega_n^\frac2n\big)^{-p}pB\big(\frac{n}{n-2}-p,p\big),
\end{eqnarray*}
producing
$$
\lim_{p\to\frac{n}{n-2}}\mathsf{G}(p;x,\Omega)\ge\lim_{r\to\infty}D_\Omega(x,r)^\frac{n}{2-n}.
$$
Therefore, the first limit formula follows.

In order to verify the second limit formula, we observe that
$G_\Omega(x,y)=r$ implies
$$
d(x,y)^{2-n}=n(n-2)\omega_n
r+R_\Omega(x)^{2-n}+o(1)\quad\hbox{as}\quad r\to\infty,
$$
and hence
$$
B_{r_{-}}(x)\subseteq \{y\in\Omega:\ G_\Omega(x,y)\ge r\}\subseteq
B_{r_+}(x),
$$
where
$$
r_{\pm}=\big(n(n-2)\omega_n r+R_\Omega(x)^{2-n}\big)^\frac1{2-n}\pm
o(1)\quad\hbox{as}\quad r\to\infty.
$$
Since the sharp isoperimetric inequality of Euclidean type is valid
for $(M^n,ds^2)$ which has nonnegative Ricci curvature, we conclude
(cf. \cite[page 244]{Heb}):
$$
V\big(B_{r_{\pm}}(x)\big)=\omega_n r_{\pm}^n.
$$
Consequently,
$$
\omega_n r_{-}^n\le V\big(\{y\in\Omega:\ G_\Omega(x,y)\ge
r\}\big)\le\omega_n r_+^n.
$$
This yields
$$
D_\Omega(x,r)^\frac{n}{2-n}=\omega_n R_\Omega(x)^n\pm
o(1)\quad\hbox{as}\quad r\to\infty.
$$
\end{proof}

\begin{remark}\label{rem2} Under the same hypothesis as in Theorem \ref{con2}, we can obtain
the sharp isoperimetric-type inequality for $0\le p<\frac{n}{n-2}$:
$$
\left(\frac{\big(n(n-2)\omega_n^\frac2n\big)^p}{pB\big(\frac{1}{n-2}-p,p\big)}\int_\Omega
G_\Omega(x,y)^p\,dV(y)\right)^\frac{n-1}{n-p(n-2)}\le(n\omega_n^\frac1n)^{-1}S(\partial\Omega),
$$
but also the non-sharp one for $-\frac{n}{n-2}<p<0$:
$$
V(\Omega)^\frac{n-p(n-2)}{n}\le
\big(n(n-2)\omega_n^\frac2n\big)^p(-p)B\big(\frac{n}{n-2}+p,-p\big)\int_\Omega
G_\Omega(x,y)^p\,dV(y),
$$
for all regular bounded domains $\Omega\subseteq M^n$ containing
$x$.

\end{remark}

\bibliographystyle{amsplain}

\end{document}